\PassOptionsToPackage{hidelinks}{hyperref}

\documentclass[11pt,reqno]{amsart}

\usepackage{etoolbox}
\usepackage{amssymb}
\usepackage{colonequals}
\usepackage[margin=1in]{geometry}
\usepackage{syntonly}  
\usepackage{mleftright}
\usepackage[colorlinks]{hyperref}
\def \llaa {\langle}
\def \rraa {\rangle}
\def \lllaaa {\left\langle}
\def \rrraaa {\right\rangle}
\DeclareMathOperator{\spann}{span}
\DeclareMathOperator{\Spann}{span}
\DeclareMathOperator{\Ad}{Ad}
\usepackage{hyperref}
\usepackage{orcidlink}
\allowdisplaybreaks

\newtheorem{theorem}{Theorem}[section]

\newtheorem{prop}[theorem]{Proposition}

\theoremstyle{definition}
\newtheorem{definition}[theorem]{Definition}

\theoremstyle{remark}
\newtheorem{remark}[theorem]{Remark}

\numberwithin{equation}{section}

\begin{document}
\newcommand{\spacing}[1]{\renewcommand{\baselinestretch}{#1}\large\normalsize}
\spacing{1.14}

\title{On Naturally Reductive  $\boldsymbol{(\alpha_1,\alpha_2)}$-Metrics}

\author {Ali Hatami Shahi\orcidlink{0009-0009-7177-7654}}

\address{Department of Pure Mathematics\\Faculty of Mathematics and Statistics\\University of \mbox{Isfahan}, Isfahan, Iran.} \email{a.hatami@sci.ui.ac.ir \hspace*{20mm}  ali.hatami.math.ah@gmail.com}

\author {Hamid Reza Salimi Moghaddam\orcidlink{0000-0001-6112-4259}}

\address{Department of Pure Mathematics\\Faculty of Mathematics and Statistics\\University of \mbox{Isfahan}, Isfahan, Iran.} \email{hr.salimi@sci.ui.ac.ir \hspace*{20mm} salimi.moghaddam@gmail.com}



\date{\today}

\begin{abstract}
In this paper, we investigate the converse of the Tan-Xu theorem, which states that the naturally reductive property of a Riemannian metric is inherited by a naturally reductive
$(\alpha_1,\alpha_2)$-metric, and we show that, under certain conditions, the converse also holds. We also examine the relationship between geodesic vector fields on homogeneous Riemannian spaces and homogeneous
$(\alpha_1,\alpha_2)$-spaces. Finally, we construct left-invariant $(\alpha_1,\alpha_2)$-metrics on the tangent bundle of Lie groups using left-invariant Randers metrics on the base Lie group, and study their geometric relations.
\end{abstract}

\maketitle
\section{\bfseries Introduction}\label{s1}
Z. Shen and S. Chern were the first to identify, while investigating product Finsler metrics \cite{6}, that there is no natural method for defining a product Finsler metric on the product manifold formed by a pair of Finsler manifolds. However, when the Finsler metrics are Riemannian, it becomes possible to construct various product Finsler metrics. In their book, the authors introduced a norm known as the $f$-product, which essentially marks the emergence of Finsler metrics referred to as $(\alpha_1,\alpha_2)$-metrics, although this term itself was not explicitly stated. Notably, the $f$-product is a specific instance of the Berwald-type $(\alpha_1, \alpha_2)$-metric.

In 2016, S. Deng and M. Xu first introduced the concept of Finsler metrics of the \((\alpha_1,\alpha_2)\) type while researching Clifford-Wolf transformations of homogeneous Finsler spaces (see \cite{11} and \cite{12}). They particularly focused on Clifford-Wolf homogeneous Randers spaces and left-invariant \((\alpha,\beta)\)-metrics of Clifford-Wolf type \cite{22,23}. Their investigate these concepts on \(f\)-products and their generalizations. They stated that \((\alpha_1,\alpha_2)\)-metrics, in a sense, can be viewed as a generalization of \((\alpha,\beta)\)-metrics. Additionally, they defined a notion of suitable normalized data for homogeneous \((\alpha_1,\alpha_2)\)-metrics, which they used to explore geometric properties. The authors also characterized the conditional Clifford-Wolf homogeneity of left-invariant \((\alpha_1,\alpha_2)\)-metrics on simply connected compact Lie groups, demonstrating that in special cases, these metrics must be Riemannian.

In 2022, L. Zhang and M. Xu introduced the concept of a standard homogeneous \((\alpha_1, \alpha_2)\)-metric, generalizing normal homogeneity in Finsler geometry while maintaining relatively good computability. They explored the geodesic orbit (g.o.) property of this metric and established a criterion for a g.o. \((\alpha_1, \alpha_2)\)-space \cite{24}.

In the next year, H. An, Z. Yan, and S. Zhang gave a complete description of invariant $(\alpha_1,\alpha_2)$-metrics on spheres with zero $S$-curvature and investigated the geodesic orbit property and $S$-curvature of these invariant metrics on spheres and found that there are many Finsler spheres with zero $S$-curvature that are not geodesic orbit spaces \cite{2}. Furthermore, in 2023, J. Tan and M. Xu explored the concept of Naturally Reductive \((\alpha_1, \alpha_2)\)-metrics from both geometric and algebraic perspectives \cite{21}. From a geometric viewpoint, they noted that every non-Riemannian naturally reductive homogeneous \((\alpha_1, \alpha_2)\)-metric is locally isometric to an \(f\)-product formed by two naturally reductive Riemannian metrics. On the algebraic side, they demonstrated that if a non-Riemannian homogeneous \((\alpha_1, \alpha_2)\)-metric is derived from a homogeneous Riemannian metric \(\alpha\), then the natural reductivity of the \((\alpha_1, \alpha_2)\)-metric is inherited from \(\alpha\). At the end, the authors presented an explicit formula for the flag curvature of a naturally reductive \((\alpha_1, \alpha_2)\)-metric of the form \(F = \alpha\phi\left(\frac{\alpha_2}{\alpha}\right)\).

As previously mentioned, the \((\alpha_1, \alpha_2)\)-metrics can be regarded as a generalization of \((\alpha, \beta)\)-metrics, suggesting that results established for \((\alpha, \beta)\)-metrics could also be extended to \((\alpha_1, \alpha_2)\)-metrics. For instance, L. Huang and X. Mo \cite{14} constructed several homogeneous Einstein metrics that fall within the \((\alpha_1, \alpha_2)\) category.

This paper aims to study naturally reductive \((\alpha_1, \alpha_2)\)-metrics on tangent Lie groups. We begin with computing the formula for the Cartan tensor associated with an \((\alpha_1, \alpha_2)\)-metric. Next, we explore the converse of the Tan-Xu theorem \cite{21}, which asserts that the naturally reductive property of a Riemannian metric is inherited by a naturally reductive \((\alpha_1, \alpha_2)\)-metric. We demonstrate that, under certain conditions, this converse also holds. Additionally, we investigate the relationship between geodesic vector fields on homogeneous Riemannian spaces and homogeneous \((\alpha_1, \alpha_2)\)-spaces. At the end, using left-invariant Randers metrics on the base Lie group, we construct left-invariant $(\alpha_1,\alpha_2)$-metrics on the tangent bundle of Lie groups and study some geometric relations between them.

\section{\bfseries Preliminaries}\label{s2}

This section provides basic definitions, notations, and results related to Finsler geometry, particularly concerning $(\alpha_1,\alpha_2)$-metrics and also the lifting of a Riemannian metric on the tangent bundle of a manifold.

Let $(M,F)$ be a Finsler manifold. We recall that the fundamental tensor and the Cartan tensor are symmetric sections of the vector bundles  $({\pi }^*T^*M)^{\otimes 2}$ and $({\pi }^*T^*M)^{\otimes 3}$, respectively, defined as follows:
\[g_y (u,v) \colonequals \frac{1}{2}\frac{{\partial }^2}{\partial s \partial t}{\left[F^2\left(y+su+tv\right)\right]}_{s=t=0}=g_{ij}(y)u^iv^j,\]
\[C_y (u,v,w)\! =\!\frac{1}{2}\frac{d}{dt}{\left[g_{y+tw}\right]}_{t=0}=\frac{1}{4}\frac{{\partial }^3}{\partial r\partial s\partial t}{\left[F^2(y+ru+sv+tw\right]}_{r=s=t=0}\!=\!c_{ijk} (y)u^iv^jw^k,\]
where $y\neq 0$, and $g_{ij}=\frac{1}{2}\frac{{\partial }^2F}{\partial y^i\partial y^j}$ and $c_{ijk}=\frac{1}{2}\frac{\partial g_{ij}}{\partial y^k}$ denote their components in local coordinates.

The spray coefficients $G_i$ of the spray vector field $G=y^i\frac{\partial }{\partial x^i}-2G^i\frac{\partial }{\partial y^i}$ on $TM\setminus \{0\}$, in a standard local coordinate system $(x^i,y^i)$ of $TM$, are defined by $G^i(x,y)=\frac{1}{4} g^{il}[{\left(F^2\right)}_{x^ky^l}y^k-{\left(F^2\right)}_{x^l}]$. If there exists a function $P(x ,y)$ which is positively $y$-homogeneous of degree one, such that $G^i=\frac{1}{2}\Gamma^i_{jk}(x)y^jy^k+P(x,y)y^i$ then the Finsler metric $F$ is called of Douglas type. In particular, if $P(x, y)=0$, then the Finsler metric $F$ is called a Berwald metric.

Randers metrics are a special class of Finsler metrics defined by the following expression:
\[F_x(y) = F(x,y) = \alpha(x,y) + \beta(x,y) = \alpha(x,y) + g(X_x, y),\]
where \(\alpha(x,y) = \sqrt{g_x(y,y)}\) represents a Riemannian norm, \(\beta\) is a one-form, \(X\) is the vector field corresponding to the one-form \(\beta\), and \(\|\beta\|_\alpha = \|X\|_\alpha < 1\) (see \cite{5}).

It can be shown that for a Randers metric \(F\), the condition that \(F\) is of Berwald type is equivalent to the condition that \(X\) is a parallel vector field with respect to the Riemannian connection \(g\) (see \cite{6}).

We will now introduce the preliminaries for $(\alpha_1, \alpha_2)$-metrics, the main metrics explored in this article.

\begin{definition}\cite{11}
\label{t2.25}
A Finsler metric $F$ on a manifold $M$ is called an $(\alpha _1,\alpha _2)$-metric if there exists a Riemannian metric $\alpha $ on $M$ and an $\alpha$-orthogonal bundle decomposition $TM={\mathcal{V}}_1 \bigoplus {\mathcal{V}}_2$ with respect to (possibly non-integrable) distributions ${\mathcal{V}}_i$ ($i=1,2$) such that for any$ x\in M$ and $y=y_1+y_2\in T_xM$, where $y_i\in {\mathcal{V}}_i$, F can be represented as:
\[F_x (y)=f(\alpha _1(x,y_1) , \alpha _2(x,y_2)),\]
in which each $\alpha _i$ is defined on a distribution ${\mathcal{V}}_i$ and $\alpha _i(y)=\alpha (y_i)={\left.\alpha \right|}_{{\mathcal{V}}_i}$ and also $f:{\mathbb{R}}^2\to \mathbb{R}$ is a positive smooth function. The pair $(n_1,n_2)$ such that $dimTM=n$, $dimV_i=n_i$ ($i=1, 2$) and $n=n_1+n_2$ is called the dimension decomposition for the $(\alpha _1,\alpha _2)$-metric.
\end{definition}
\begin{remark}
An $(\alpha _1,\alpha _2)$-metric $F$ can be written as $F= \alpha \phi (\frac{\alpha _2}{\alpha} )$ (or $F=\alpha \psi (\frac{\alpha _1}\alpha )$) where $\phi (s)$ $(\psi (s))$ is a positive smooth function on $(0,1)$ and $\phi (s)=\psi \left(\sqrt{1-s^{2}}\right)$. Sometimes it is necessary to write an $(\alpha _1,\alpha _2)$-metric as $F=\sqrt{L(\alpha^2_1,\alpha^2_2)}$ where $L\colon {\mathbb{R}}^+\cup \{0\}\times {\mathbb{R}}^+\cup \{0\}\to {\mathbb{R}}^+\cup \{0\}$ is a positive smooth real function homogeneous of degree 1. With a little calculation, we have $L(a,b)=\left(a+b\right)\phi^2\left(\sqrt{\frac{b}{a+b}}\right)=\left(a+b\right)\psi^2\left(\sqrt{\frac{a}{a+b}}\right)$ where $a,b\in {\mathbb{R}}^+\times {\mathbb{R}}^+$.
\end{remark}
\begin{remark}
\label{tr.2}
(i)
For the cases $n_2=0$ and $n_2=1$, the ($\alpha_1, \alpha_2$) metric $F$ is either an Euclidean metric or an $(\alpha , \beta)$-metric, respectively.
\\
(ii)An $(\alpha _1,\alpha _2)$-metric $F$ is Riemannian if and only if for positive constants $k, l>0$, $F$ has the form $f(u_1,u_2)=\sqrt{ku^2_1+lu^2_2}$.
\end{remark}

\begin{theorem}\cite{11}
\label{t2.26}
Assume $\phi $ and $\psi $ are two positive functions on the interval $[0,1]$ such that $\phi (s) = \psi \left(\sqrt{1-s^2}\right)$. Also, let $F=\alpha \phi (\frac{\alpha _2}{\alpha} )=\alpha \psi (\frac{\alpha _1}\alpha )$ be a $(\alpha _1,\alpha _2)$-norm on $\mathbb{R}^n$ with dimension decomposition $(n_1, n_2)$ such that $n_1\ge  n_2\ge 1$. Then $\phi $ and $\psi $ are both positive smooth functions on $[0,1]$, and for any $s$ and $b$ with the condition $0\le  s\le b\le 1$, we have
\begin{align*}
& \phi(s)-s\phi'(s)+(b^2-s^2)\phi''(s)>0, \\
& \psi(s)-s\psi'(s)+(b^2-s^2)\psi''(s)>0.
\end{align*}
Conversely, if $\phi (s)-s\phi '(s)+(1- s^2)\phi ''(s)>0$, then $F(\alpha )=\alpha \phi (\frac{\alpha _1}\alpha )$ defines a $(\alpha _1,\alpha _2)$-norm with dimension decomposition $(n_1, n_2)$ where $n_1\ge n_2\ge 1$.
\end{theorem}

There are two definitions for a naturally reductive Finsler space as follows.
\begin{definition}\cite{17}
\label{t2.17}
A homogeneous manifold $M=G/H$ with an invariant Finsler metric $F$ is called naturally reductive if there exists an $\Ad (H)$-invariant decomposition $\mathfrak{g}=\mathfrak{m}+\mathfrak{h}$ such that
\[g_y\left(u,{\left[z,x\right]}_{\mathfrak{m}}\right)+g_y\left({[z,u]}_{\mathfrak{m}},x\right)+2C_y\left( [z,y] ,u,x\right)=0,\]
where $y\neq 0$, $u$, $u,x,z\in \mathfrak{m}$, and $C_y$ is the Cartan tensor at $y$.
\end{definition}
\begin{definition}\cite{9,10}
\label{t2.18}
A homogeneous manifold $M$ equipped with an invariant Finsler metric $F$ is called naturally reductive if there exists an invariant Riemannian metric $g$ on $M$ such that the space $(M=G/H,g)$ is naturally reductive and the Riemannian connection and the Chern connection coincide.
\end{definition}
 Zhang, Yan, and Deng proved in \cite{25} that these two definitions are equivalent. Also, according to \cite{10} and \cite{13}, another equivalent definition of a naturally reductive homogeneous Finsler space is
 presented in terms of the corresponding spray vector field \cite{21}.

Xu and Tan derived in their paper \cite{21}  the fundamental tensor and flag curvature for a naturally reductive metric of the form $F=\alpha \phi \left(\frac{\alpha _2}{\alpha} \right)= |y| \phi \left(\frac{|y_2|}{|y|}\right)$
($F=\alpha \psi \left(\frac{\alpha _1}{\alpha} \right)= |y| \psi \left(\frac{|y_1|}{|y|}\right)$),
where $ |\cdot| =\llaa\cdot,\cdot \rraa^{\frac{1}{2}} $ is the $\Ad (H)$-invariant Euclidean norm on $\mathfrak{m}$ and $y=y_1+y_2\in \mathfrak{m}$ with $y_i\in {\mathfrak{m}}_i$, for an $\Ad (H)$-invariant $\left\langle  \cdot,\cdot \right\rangle $-orthogonal decomposition $\mathfrak{m}={\mathfrak{m}}_1+{\mathfrak{m}}_2$, on the manifold $M=G/H$ with respect to the reductive decomposition $\mathfrak{g}=\mathfrak{m}+\mathfrak{h}$.
\begin{align*}
g_y (u,v)  & =\left.\frac{1}{2}\frac{{\partial }^2}{\partial s\partial t}F^2\left(y+su+tv\right)\right|_{s=t=0}
\\
&  = \langle u,v\rangle \phi^2(b)
\\
& + \phi (b) {\phi}'(b)\left(\frac{\llaa y,v\rraa  \llaa  y_2,u_2\rraa  }{ |y|   |y_2| }+\frac{\llaa  y_2,v_2 \rraa  \llaa  y,u\rraa  }{ |y|   |y_2| }  -\frac{\llaa  y,u\rraa  \llaa y,v\rraa   |y_2| }{{|y|}^3} +\frac{\llaa  u_2,v_2 \rraa   |y| }{|y_2|} \right.
\\
&\quad  \left. -\frac{\llaa  u,v\rraa    |y_2| }{|y|}-\frac{\llaa  u_2,y_2\rraa   \llaa  v_2,y_2 \rraa    |y| }{{|y_2|}^3}\right)
\\
& +\left(\phi^{'^2}(b)+\phi (b)\phi''(b)\right)\left(\frac{\llaa  u_2,y_2\rraa  }{ |y|   |y_2| }-\frac{\llaa u,y \rraa    |y_2| }{{|y|}^3}\right)\left(\frac{ \llaa  v_2,y_2 \rraa    |y| }{|y_2|}-\frac{ \llaa  v,y \rraa   |y_2|}{|y|}\right)
\\
&                = \llaa  u,v\rraa   \psi^2 (b_1)
\\
& + \psi  (b_1)  \psi' (b_1) \left(\frac{\llaa y,v\rraa  \llaa  y_1,u_1\rraa   }{ |y|   |y_1| }+\frac{\llaa  y_1,v_1\rraa   \llaa  y,u\rraa  }{ |y|   |y_1| } -\frac{\llaa  y,u\rraa  \llaa y,v\rraa   |y_1| }{{|y|}^3}  +\frac{\llaa  u_1,v_1 \rraa    |y| }{|y_1|}
\right.
\\
& \quad  \left.-\frac{\llaa  u,v\rraa    |y_1| }{|y|}-\frac{\llaa  u_1,y_1\rraa  \llaa  v_1,y_1 \rraa    |y| }{{|y_1|}^3}\right)
\\
& +\left(\psi^{'^2} (b_1) +\psi (b_1)\psi''(b_1)\right)\left(\frac{\llaa  u_1,y_1\rraa  }{ |y|   |y_1| }-\frac{\llaa u,y \rraa    |y_1| }{{|y|}^3}\right)\left(\frac{\llaa  v_1,y_1 \rraa    |y| }{|y_1|}-\frac{ \llaa  v,y \rraa   |y_1|}{|y|}\right),
\end{align*}
where $b=\frac{|y_2|}{|y|}$ and $b_1=\frac{|y_1|}{|y|}$.

Also, with similar calculations, the formula for the fundamental tensor of a $\left(\alpha _1,\alpha _2\right)$-metric of the form $F=\sqrt{L\left(\alpha^2_1,\alpha^2_2\right)}$ is obtained as follows:
\begin{align*}
g_y (u,v)
& =\frac{1}{2}\frac{{\partial }^2}{\partial s\partial t}{\left.F^2\left(y+su+tv\right)\right|}_{s=t=0}
\\
& =\frac{1}{2}\frac{{\partial }^2}{\partial s\partial t} \left.L\left(\alpha^2_1\left(y_1+su_1+tv_1\right),\alpha^2_2\left(y_2+su_2+tv_2\right)\right)\right|_{s=t=0}
\\
& =\Big [L_{11}\left({|y_1|}^2,{|y_2|}^2\right)\llaa  u_1,y_1\rraa  \llaa  v_1,y_1 \rraa   +L_{21}\left({|y_1|}^2,{|y_2|}^2\right)\llaa  u_2,y_2\rraa  \llaa  v_1,y_1 \rraa
\\
&+ L_{12}\left({|y_1|}^2,{|y_2|}^2\right)\llaa  u_1,y_1\rraa   \llaa  v_2,y_2 \rraa   +L_{22}\left({|y_1|}^2,{|y_2|}^2\right)\llaa  u_2,y_2\rraa   \llaa  v_2,y_2 \rraa   \Big]
\\
&                     + \left[L_1\left({|y_1|}^2,{|y_2|}^2\right)\llaa  u_1,v_1 \rraa   +L_2\left({|y_1|}^2,{|y_2|}^2\right)\llaa  u_2,v_2 \rraa  \right],
\end{align*}
where $L_1(\cdot,\cdot)$, $L_2(\cdot,\cdot)$, $L_{11}(\cdot,\cdot)$, etc., are partial derivatives with respect to the variables indicated by the lower indices and also $u_i$ and $v_i$ are the components of $u$ and $v$ in ${\mathfrak{m}}_i$, respectively.

Since we want to work on the tangent Lie group $TG$, here we mention some preliminaries about such spaces.
\begin{theorem}\cite{3}
\label{t2.13}
Let $X$ and $Y$ be two left-invariant vector fields on a Lie group $G$, and $X^c$ and $Y^c$, and $X^v$ and $Y^v$ denote their corresponding complete and vertical lifts on the tangent bundle $TG$, respectively. Then for their Lie bracket we have:
\begin{align*}
& \left[X^c,Y^c\right]={\left[X,Y\right]}^c,  \\
& \left[X^c,Y^v\right]={\left[X,Y\right]}^v,  \\
& \left[X^v,Y^v\right]=0.
\end{align*}
\end{theorem}
\begin{theorem}\cite{3}
\label{t2.14}
If $\{X_1, X_2, \ldots , X_n\}$ is a basis consisting of left-invariant vector fields for the Lie algebra of a Lie group $G$, then $\{X^c_1, X^v_1, X^c_2, X^v_2, \ldots , X^c_n, X^v_n\}$ is a basis consisting of left-invariant vector fields for the Lie algebra of the tangent Lie group $TG$.
\end{theorem}
\begin{definition}\cite{3}
\label{t2.15}
Let $g$ be a left-invariant Riemannian metric on a Lie group $G$. A natural Riemannian metric $\tilde{g}$ on the tangent bundle $TG$ is defined as follows, and is clearly left-invariant:
\begin{align*}
& \widetilde{g}\left(X^c, Y^c\right)=g\left(X,Y\right),\\
& \widetilde{g}\left(X^v, Y^v\right)=g\left(X,Y\right),\\
& \widetilde{g}\left(X^c, Y^v\right)=0.
\end{align*}
\end{definition}

\section{\bfseries On the Natural Reductivity of $\boldsymbol{(\alpha_1,\alpha_2)}$-Metrics}
As we have seen in Definition \ref{t2.17}, besides the fundamental tensor, the Cartan tensor is an important quantity to study the natural reductivity of a homogeneous Finsler space. So, at first, we derive the formula for the Cartan tensor of an $(\alpha _1,\alpha _2)$-metric of the form $F=\alpha \phi \left(\frac{\alpha _2}{\alpha} \right)= |y| \phi \left(\frac{|y_2|}{|y|}\right)$ on a manifold $M=G/H$ with respect to the reductive decomposition $\mathfrak{g}=\mathfrak{m}+\mathfrak{h}$, where$ y\in \mathfrak{m}\setminus \{0\}$ and $u,v,w\in \mathfrak{m}$ such that $y=y_1+y_2$ and $u=u_1+u_2$ and $v=v_1+v_2$ and $w=w_1+w_2$ with $y_i,u_i,v_i,w_i\in {\mathfrak{m}}_i$ ($i=1,2$) and $\mathfrak{m}={\mathfrak{m}}_1+{\mathfrak{m}}_2$ is an $\Ad (H)$-invariant $\llaa\cdot,\cdot\rraa$-orthogonal, and $ |\cdot| =\langle\cdot,\cdot\rangle^{\frac{1}{2}} $ is an $\Ad (H)$-invariant Euclidean norm on $\mathfrak{m}$.
\begin{align*}
C_y (u,v,w)
 & =\left.\frac{1}{2}\frac{d}{dt}\right|_{t=0} g_{y+tw} (u,v) =\frac{1}{4} \left.\frac{{\partial }^3}{\partial r\partial s\partial t}\right|_{r=s=t=0}F^2\left(y+ru+sv+tw\right)\\
& =\frac{1}{4}{\left.\frac{{\partial }^3}{\partial r\partial s\partial t}\right|}_{r=s=t=0}\left\langle y+ru+sv+tw,y+ru+sv+tw\right\rangle
\\
& \times  \phi^2\left(\frac{\lllaaa{y_2+ru_2+sv_2+tw_2,y_2+ru_2+sv_2+tw_2}\rrraaa^{\frac{1}{2}}}{\lllaaa{y+ru+sv+tw,y+ru+sv+tw}\rrraaa^{\frac{1}{2}}}\right).
\end{align*}

To simplify the calculations, we set $K=F^2\left(y+ru+sv+tw\right)$, $A=y+ru+sv+tw$, $B=y_2+ru_2+sv_2+tw_2$, $M=\left\langle y+ru+sv+tw,y+ru+sv+tw\right\rangle$, and \\
$N  =\frac{\lllaaa{y_2+ru_2+sv_2+tw_2,y_2+ru_2+sv_2+tw_2}\rrraaa^{\frac{1}{2}}}{\lllaaa{y+ru+sv+tw,y+ru+sv+tw}\rrraaa^{\frac{1}{2}}}$.
\begin{align*}
& \frac{1}{2}\frac{\partial }{\partial s}\frac{\partial K}{\partial t}
=\frac{1}{2}\frac{\partial }{\partial s}\left(\frac{\partial M}{\partial t} \phi^2 (N) +2M\phi  (N) \phi' (N) \frac{\partial N}{\partial t}\right)
\\
& =\frac{1}{2}\frac{\partial }{\partial s} \left(2\lllaaa\frac{\partial A}{\partial t}, A \rrraaa\phi^2 (N) +2 \llaa  A,A \rraa   \phi  (N) \phi' (N) \frac{\partial N}{\partial t}\right)
\\
& = \frac{\partial }{\partial s} \left(\lllaaa\frac{\partial A}{\partial t},A\rrraaa\right)\phi^2 (N) +2\lllaaa\frac{\partial A}{\partial t},A\rrraaa\phi  (N) \phi' (N) \frac{\partial N}{\partial s}
\\
&  \quad    +\frac{\partial }{\partial s} \left[\phi  (N) \phi' (N) \left(\frac{\llaa w_2,B\rraa  \llaa  A,A \rraa^{\frac{1}{2}}}{ \llaa  B,B \rraa^{\frac{1}{2}}}-\frac{\llaa w,A \rraa  \llaa  B,B \rraa^{\frac{1}{2}}}{ \llaa  A,A \rraa^{\frac{1}{2}}}\right)\right]
\\
& = \llaa v,w \rraa \phi^2 (N) +2 \llaa w,A \rraa \phi  (N) \phi' (N)  (E-F) +\frac{\partial }{\partial s}(\phi  (N) \phi' (N)  (C-D) )
\\
&=\Bigg(\llaa v,w \rraa \phi^2 (N) +2 \llaa w,A \rraa \phi  (N) \phi' (N)  (E-F) +{\phi'}^2 (N) \frac{\partial N}{\partial s} (C-D)
\\
&    \left. \quad  +\phi  (N) \phi'' (N) \frac{\partial N}{\partial s} (C-D) +\phi  (N) \phi'(N)(\frac{\partial C}{\partial s}-\frac{\partial D}{\partial D}\right)\Bigg),
\end{align*}
where $C=\frac{\llaa w_2,B \rraa  \llaa  A,A \rraa^{\frac{1}{2}}}{ \llaa  B,B \rraa^{\frac{1}{2}}}$, $D=\frac{\llaa w,A\rraa  \llaa  B,B \rraa^{\frac{1}{2}}}{ \llaa  A,A \rraa^{\frac{1}{2}}}$, $E=\frac{\llaa v_2,B \rraa }{ \llaa  A,A \rraa^{\frac{1}{2}} \llaa  B,B \rraa^{\frac{1}{2}}}$, and $F=\frac{ \llaa  v,A \rraa    \llaa  B,B \rraa^{\frac{1}{2}}}{ \llaa  A,A \rraa^{\frac{3}{2}}}$.

So we have:
\begin{align*}
\frac{1}{2}\frac{\partial }{\partial r}\left(\frac{{\partial }^2K}{\partial s\partial t}\right)
&=\frac{1}{2}\frac{\partial }{\partial r}\Bigg [ \llaa   v,w \rraa   \phi^2 (N) +2\left\langle w,A\right\rangle \phi  (N) \phi' (N)  (E-F) +\phi^{'^2}  (N)  (C-D)  (E-F)
\\
&\quad +\phi  (N) \phi'' (N)  (C-D)  (E-F) +\phi  (N) \phi' (N) \left(\frac{\partial C}{\partial s}-\frac{\partial D}{\partial s}\right)\Bigg]
\\
&                     = \llaa   v,w \rraa   \phi  (N) \phi' (N) \frac{\partial N}{\partial r}+\left\langle u,w\right\rangle \phi  (N) \phi' (N)  (E-F) +\left\langle w,A\right\rangle {\phi'}^2 (N)  (E-F)
\\
&                    \quad               +\left\langle w,A\right\rangle \phi  (N) \phi'' (N) \frac{\partial N}{\partial r} (E-F) +\left\langle w,A\right\rangle \phi  (N) \phi' (N) \left(\frac{\partial E}{\partial r}-\frac{\partial F}{\partial r}\right)
\\
&                    \quad               +\phi' (N) \phi'' (N) \frac{\partial N}{\partial r} (C-D)  (E-F) +\frac{1}{2}{\phi'}^2 (N) \left(\frac{\partial C}{\partial r}-\frac{\partial D}{\partial r}\right) (E-F) \\
&                                  +\frac{1}{2}{\phi'}^2 (N)  (C-D) \left(\frac{\partial E}{\partial r}-\frac{\partial F}{\partial r}\right)+\frac{1}{2}\phi' (N) \phi'' (N) \frac{\partial N}{\partial r} (C-D)  (E-F)
\\
& \quad                                   +\frac{1}{2}\phi  (N) \phi^{'''} (N) \frac{\partial N}{\partial r} (C-D)  (E-F) +\frac{1}{2}\phi  (N) \phi'' (N) \left(\frac{\partial C}{\partial r}-\frac{\partial D}{\partial r}\right) (E-F)
\\
& \quad                                   +\frac{1}{2}\phi  (N) \phi'' (N)  (C-D) \left(\frac{\partial E}{\partial r}-\frac{\partial F}{\partial r}\right)+\frac{1}{2}{\phi'}^2 (N) \frac{\partial N}{\partial r}\left(\frac{\partial C}{\partial s}-\frac{\partial D}{\partial s}\right)
\\
& \quad                                   +\frac{1}{2}\phi  (N) \phi'' (N) \frac{\partial N}{\partial r}\left(\frac{\partial C}{\partial s}-\frac{\partial D}{\partial s}\right)+\frac{1}{2}\phi  (N) \phi'(N)(\frac{{\partial }^2C}{\partial r\partial s}-\frac{{\partial }^2D}{\partial r\partial s})
\\
&  =\phi  (N) \phi' (N)\Bigg [ \llaa   v,w \rraa   \frac{\partial N}{\partial r}+\left\langle u,w\right\rangle  (E-F) +\left\langle w,A\right\rangle \left(\frac{\partial E}{\partial r}-\frac{\partial F}{\partial r}\right)
\\
& \quad +\frac{1}{2}\left(\frac{{\partial }^2C}{\partial r\partial s}-\frac{{\partial }^2D}{\partial r\partial s}\right)\Bigg]
\\
& \quad     +(\phi^{'^2} (N) +\phi  (N) \phi'' (N) ) \Bigg [\left\langle w,A\right\rangle \frac{\partial N}{\partial r} (E-F) +\frac{1}{2}\left(\frac{\partial C}{\partial r}-\frac{\partial D}{\partial r} \right) (E-F)
\\
& \quad +\frac{1}{2} (C-D) \left(\frac{\partial E}{\partial r}-\frac{\partial F}{\partial r}\right)
 +\frac{1}{2}\frac{\partial N}{\partial r}\left(\frac{\partial C}{\partial s}-\frac{\partial D}{\partial s}\right)\Bigg]
\\
& \quad    +\frac{1}{2}\left(3\phi' (N) \phi'' (N) +\phi  (N) \phi^{'''} (N) \right)\left[\frac{\partial N}{\partial r} (C-D)  (E-F) \right].
\end{align*}

Hence for the Cartan tensor we have:
\begin{align*}
{}&{}C_y (u,v,w)
 =\frac{1}{4} {\left.\frac{{\partial }^3}{\partial r\partial s\partial t}\right|}_{r=s=t=0}F^2(y+ru+sv+tw)                                                                              \\
& =\phi (b)\phi'(b) \Bigg[ \llaa   v,w \rraa   \left(\frac{\left\langle u_2,y_2\right\rangle }{ |y_2|  |y| }-\frac{\left\langle u,y\right\rangle  |y_2| }{{|y|}^3}\right)+\left\langle u,w\right\rangle \left(\frac{ \llaa  v_2,y_2 \rraa    }{ |y_2|  |y| }-\frac{ \llaa   v,y \rraa     |y_2| }{{|y|}^3}\right)
\\
& \quad +\langle w,y\rangle \Bigg(\frac{\left\langle u_2,v_2\right\rangle }{ |y_2|  |y| }-\frac{ \llaa  v_2,y_2 \rraa    \left\langle u,y\right\rangle }{ |y_2| {|y|}^3}-\frac{ \llaa  v_2,y_2 \rraa    \left\langle u_2,y_2\right\rangle }{ |y| {|y_2|}^3}-\frac{\left\langle u,v\right\rangle  |y_2| }{{|y|}^3}-\frac{ \llaa   v,y \rraa    \left\langle u_2,y_2\right\rangle }{ |y_2| {|y|}^3}
\\
& \qquad +\frac{3\left\langle u,y\right\rangle  \llaa   v,y \rraa     |y_2| }{{|y|}^5}\Bigg)
\\
& \quad    +\frac{1}{2}\Bigg (\frac{2\left\langle w_2,v_2\right\rangle \left\langle u,y\right\rangle +\left\langle u,v\right\rangle \left\langle w_2,y_2\right\rangle + \llaa   v,y \rraa    \left\langle u_2,w_2\right\rangle }{ |y_2|  |y| }
\\
& \qquad -\frac{\left(\left\langle u,y\right\rangle {|y_2|}^2+\left\langle u_2,y_2\right\rangle {|y|}^2\right)\left(\left\langle w_2,v_2\right\rangle {|y|}^2+ \llaa   v,y \rraa    \left\langle w_2,y_2\right\rangle \right)}{{|y_2|}^3{|y|}^3}
\\
& \quad                                 -\frac{2\left\langle w,v\right\rangle \left\langle u_2,y_2\right\rangle +\left\langle u_2,v_2\right\rangle \left\langle w,y\right\rangle + \llaa  v_2,y_2 \rraa    \left\langle u,w\right\rangle }{ |y_2|  |y| }
\\
& \qquad +\frac{\left(\left\langle u,y\right\rangle {|y_2|}^2+\left\langle u_2,y_2\right\rangle {|y|}^2\right)\left(\left\langle w,v\right\rangle {|y_2|}^2+ \llaa  v_2,y_2 \rraa    \left\langle w,y\right\rangle \right)}{{|y_2|}^3{|y|}^3}
\\
& \quad     -\frac{\left(\left\langle u_2,v_2\right\rangle \left\langle w_2,y_2\right\rangle + \llaa  v_2,y_2 \rraa    \left\langle u_2,w_2\right\rangle \right) |y| }{{|y_2|}^3}-\frac{ \llaa  v_2,y_2 \rraa    \left\langle w_2, y_2\right\rangle \left\langle u,y\right\rangle }{{ |y|  |y_2| }^3}
\\
& \qquad +\frac{3 |y| \left\langle u_2,y_2\right\rangle  \llaa  v_2,y_2 \rraa    \left\langle w_2,y_2\right\rangle }{{|y_2|}^5}
\\
&  \quad        +\frac{\left(\left\langle u,v\right\rangle \left\langle w,y\right\rangle + \llaa   v,y \rraa    \left\langle u,w\right\rangle \right) |y_2| }{{|y|}^3}+\frac{ \llaa   v,y \rraa    \left\langle w, y\right\rangle \left\langle u_2,y_2\right\rangle }{{ |y_2|  |y| }^3}-\frac{3 |y_2| \left\langle u,y\right\rangle  \llaa   v,y \rraa    \left\langle w,y\right\rangle }{{|y|}^5}\Bigg) \Bigg]
\\
& \quad +({\phi'}^2(b)+\phi (b)\phi''(b))\Bigg[\left\langle w,y\right\rangle \left(\frac{\left\langle u_2,y_2\right\rangle }{ |y_2|  |y| }-\frac{\left\langle u,y\right\rangle  |y_2| }{{|y|}^3}\right)\left(\frac{ \llaa  v_2,y_2 \rraa    }{ |y_2|  |y| }-\frac{ \llaa   v,y \rraa     |y_2| }{{|y|}^3}\right)
\\
&    +\frac{1}{2}\Bigg(\frac{\left\langle u_2,w_2\right\rangle  |y| }{|y_2|}+\frac{\left\langle u,y\right\rangle \left\langle w_2,y_2\right\rangle }{ |y_2|  |y| }-\frac{\left\langle u_2,y_2\right\rangle \left\langle w_2,y_2\right\rangle  |y| }{{|y_2|}^3} -\frac{\left\langle u,w\right\rangle  |y_2| }{|y|}-\frac{\left\langle u_2,y_2\right\rangle \left\langle w,y\right\rangle }{ |y_2|  |y| }
\\
 & \qquad +\frac{\left\langle u,y\right\rangle \left\langle w,y\right\rangle  |y_2| }{{|y|}^3}\Bigg)\Bigg(\frac{ \llaa  v_2,y_2 \rraa    }{ |y_2|  |y| }-\frac{ \llaa   v,y \rraa     |y_2| }{{|y|}^3}\Bigg)
 \\
&    +\frac{1}{2}\left (\frac{\left\langle w_2,y_2\right\rangle  |y| }{|y_2|}-\frac{\left\langle w,y\right\rangle  |y_2| }{|y|}\right)
\Bigg(\frac{\left\langle u_2,v_2\right\rangle }{ |y_2|  |y| }-\frac{ \llaa  v_2,y_2 \rraa    \left\langle u,y\right\rangle }{ |y_2| {|y|}^3}-\frac{ \llaa  v_2,y_2 \rraa    \left\langle u_2,y_2\right\rangle }{ |y| {|y_2|}^3}
\\
&  \qquad  -\frac{\left\langle u,v\right\rangle  |y_2| }{{|y|}^3}-\frac{ \llaa   v,y \rraa    \left\langle u_2,y_2\right\rangle }{ |y_2| {|y|}^3}+\frac{3\left\langle u,y\right\rangle  \llaa   v,y \rraa     |y_2| }{{|y|}^5}\Bigg)
\\
 &   +\frac{1}{2}\left(\frac{\left\langle u_2,y_2\right\rangle }{ |y_2|  |y| }-\frac{\left\langle u,y\right\rangle  |y_2| }{{|y|}^3}\right)
 \Bigg(\frac{\left\langle w_2,v_2\right\rangle  |y| }{|y_2|}+\frac{ \llaa   v,y \rraa    \left\langle w_2,y_2\right\rangle }{ |y_2|  |y| }-\frac{ \llaa  v_2,y_2 \rraa    \left\langle w_2,y_2\right\rangle  |y| }{{|y_2|}^3}
 \\
 & \qquad -\frac{\left\langle w,v\right\rangle  |y_2| }{|y|}
 -\frac{\left\langle v_2, y_2\right\rangle \left\langle w,y\right\rangle }{ |y_2|  |y| }+\frac{ \llaa   v,y \rraa    \left\langle w,y\right\rangle  |y_2| }{{|y|}^3}\Bigg)\Bigg]
\\
& \quad     +\frac{1}{2}\left(3\phi'(b)\phi''(b)+\phi (b)\phi^{'''}(b)\right)
\Bigg[\left(\frac{\left\langle u_2,y_2\right\rangle }{ |y|  |y_2| }-\frac{\left\langle u,y\right\rangle  |y_2| }{{|y|}^3}\right)\left(\frac{\left\langle w_2,y_2\right\rangle  |y| }{|y_2|}
-\frac{\left\langle w,y\right\rangle  |y_2| }{|y|}\right)
\\
& \qquad \quad \left(\frac{ \llaa  v_2,y_2 \rraa     |y| }{|y_2|}
 -\frac{ \llaa   v,y \rraa     |y_2| }{|y|}\right)\Bigg],
 \end{align*}
in which $b=\frac{|y_2|}{|y|}$.

We note that in  \cite{1}, it has been shown under certain conditions that the tangent Lie groups are naturally reductive. Also, in  \cite{21}, it has been proven that if a (non-Riemannian) homogeneous $(\alpha _1,\alpha_2)$-metric $F$ on $M=G/H$ is naturally reductive with respect to the reductive decomposition $\mathfrak{g}=\mathfrak{m}+\mathfrak{h}$, then the Riemannian metric $\alpha =\sqrt{\alpha^2_1+\alpha^2_2}$ is also naturally reductive with respect to the same decomposition. The following theorem states under what conditions the converse of the latter proposition holds.
\begin{theorem}
\label{t4.1}
Suppose $\left(M=G/H,F\right)$ is a homogeneous Finsler space where $F$ is an invariant $(\alpha_1, \alpha_2)$-metric defined by the invariant Riemannian metric $\alpha =\sqrt{g(\cdot,\cdot)}$. If $(M, g)$ is naturally reductive and ${\mathfrak{g}}'=[\mathfrak{g},\mathfrak{g}]\subseteq \mathfrak{h}+{\mathfrak{m}}_1$, then $(M, F)$ is also naturally reductive.
\end{theorem}

\begin{proof}
Assume $(M, g)$ is naturally reductive. We show that:
\[g_y\left({[z,u]}_{\mathfrak{m}},v\right)+g_y\left(u,{[z,v]}_{\mathfrak{m}}\right)+2C_y\left({[z,y]}_{\mathfrak{m}},u,v\right)=0, \qquad       y\neq 0  , z,u,v\in \mathfrak{m}\]
First, by calculating each term of the left side expression above, we observe that

\begin{align*}
g_y\left({[z,u]}_{\mathfrak{m}},v\right)
& =\left\langle {[z,u]}_{\mathfrak{m}},v\right\rangle \phi^2(b)+\phi (b)\phi'(b)
\Bigg(\frac{\left\langle y,v\right\rangle \left\langle y_2,{[z,u]}_{{\mathfrak{m}}_2}\right\rangle }{ |y|   |y_2| }+\frac{\left\langle y_2,v_2\right\rangle \left\langle y,{[z,u]}_{\mathfrak{m}}\right\rangle }{ |y|   |y_2| }
\\
& \quad -\frac{\left\langle y,{[z,u]}_{\mathfrak{m}}\right\rangle \left\langle y,v\right\rangle  |y_2| }{{|y|}^3}+\frac{\left\langle {[z,u]}_{{\mathfrak{m}}_2},v_2\right\rangle  |y| }{|y_2|}-\frac{\left\langle {[z,u]}_{\mathfrak{m}},v\right\rangle  |y_2| }{|y|}
\\
& \quad -\frac{\left\langle {[z,u]}_{{\mathfrak{m}}_2},y_2\right\rangle  \llaa  v_2,y_2 \rraa     |y| }{{|y_2|}^3}\Bigg)
\\
& \quad  +\left(\phi^{'^2}(b)+\phi (b)\phi''(b)\right)\left(\frac{\left\langle {[z,u]}_{{\mathfrak{m}}_2},y_2\right\rangle }{ |y|   |y_2| }-\frac{\left\langle {[z,u]}_{\mathfrak{m}},y\right\rangle  |y_2| }{{|y|}^3}\right) \\
& \qquad \left(\frac{ \llaa  v_2,y_2 \rraa     |y| }{|y_2|}-\frac{ \llaa   v,y \rraa     |y_2| }{|y|}\right),
\end{align*}

\begin{align*}
g_y\left({[z,v]}_{\mathfrak{m}},u\right)
& =-\left\langle {[z,u]}_{\mathfrak{m}},v\right\rangle \phi^2(b)+\phi (b)\phi'(b) \Bigg(\frac{\left\langle y,u\right\rangle \left\langle y,{{[z,v]}_{\mathfrak{m}}}_2\right\rangle }{ |y|   |y_2| }+\frac{\left\langle y_2,u_2\right\rangle \left\langle y,{[z,v]}_{\mathfrak{m}}\right\rangle }{ |y|   |y_2| }
\\
& \quad -\frac{\left\langle y,{[z,v]}_{\mathfrak{m}}\right\rangle \left\langle y,u\right\rangle  |y_2| }{{|y|}^3}+\frac{\left\langle {[z,v]}_{{\mathfrak{m}}_2},u_2\right\rangle  |y| }{|y_2|}-\frac{\left\langle {[z,v]}_{\mathfrak{m}},u\right\rangle  |y_2| }{|y|}
\\
& \quad -\frac{\left\langle {[z,v]}_{{\mathfrak{m}}_2},y_2\right\rangle \left\langle u_2,y_2\right\rangle  |y| }{{|y_2|}^3}\Bigg)
\\
&  \quad +\left(\phi^{'^2}(b)+\phi (b)\phi''(b)\right)
\!\!\mleft(\frac{\left\langle {[z,v]}_{{\mathfrak{m}}_2},y_2\right\rangle }{ |y|   |y_2| }\!-\!\frac{\left\langle {[z,v]}_{\mathfrak{m}},y\right\rangle  |y_2| }{{|y|}^3}\mright)
\!\!\mleft(\frac{\left\langle u_2,y_2\right\rangle  |y| }{|y_2|}-\frac{\left\langle u,y\right\rangle  |y_2| }{|y|}\mright),
\end{align*}
and
\begin{align*}
2C_y\left({[z,y]}_{\mathfrak{m}},u,v\right)
& =2\phi (b)\phi'(b)\left[\left\langle {[z,y]}_{\mathfrak{m}},v\right\rangle \left(\frac{\left\langle u_2,y_2\right\rangle }{ |y|   |y_2| }-\frac{\left\langle u,y\right\rangle  |y_2| }{{|y|}^3}\right)\right.
\\
& \quad + \llaa   v,y \rraa    \left(\frac{\left\langle {[z,y]}_{{\mathfrak{m}}_2},u_2\right\rangle }{ |y|  |y_2| }-\frac{\left\langle {[z,y]}_{\mathfrak{m}},u\right\rangle  |y_2| }{{|y|}^3}\right)
\\
& \quad +\frac{1}{2} \Bigg (\frac{\left\langle {[z,y]}_{\mathfrak{m}},u\right\rangle  \llaa  v_2,y_2 \rraa    }{ |y|  |y_2| }+\frac{\left\langle {[z,y]}_{{\mathfrak{m}}_2},v_2\right\rangle \left\langle u,y\right\rangle }{ |y|  |y_2| }-\frac{\left\langle {[z,y]}_{{\mathfrak{m}}_2},u_2\right\rangle  \llaa   v,y \rraa    }{ |y|  |y_2| }
\\
& \quad -\frac{\left\langle {[z,y]}_{\mathfrak{m}},v\right\rangle }{ |y|  |y_2| }-\frac{\left\langle {[z,y]}_{{\mathfrak{m}}_2},u_2\right\rangle  \llaa  v_2,y_2 \rraa     |y| }{{|y_2|}^3}-\frac{\left\langle {[z,y]}_{{\mathfrak{m}}_2},v_2\right\rangle \left\langle u_2,y_2\right\rangle  |y| }{{|y_2|}^3}
\\
& \quad \left.\left.+\frac{\langle {[z,y]}_{{\mathfrak{m}}_2},u\rangle  \llaa   v,y \rraa     |y_2| }{{|y|}^3}+\frac{\langle {[z,y]}_{\mathfrak{m}},v\rangle \langle u,y\rangle  |y_2| }{{|y|}^3}\right)\right]+2\left(\phi^{'^2}(b)+\phi (b)\phi''(b)\right)
\\
& \qquad \Bigg[\frac{1}{2} \left(\frac{\left\langle {[z,y]}_{{\mathfrak{m}}_2},v_2\right\rangle  |y| }{|y_2|}-\frac{\left\langle {[z,y]}_{\mathfrak{m}},v\right\rangle  |y_2| }{|y|}\right)\left(\frac{\left\langle u_2,y_2\right\rangle }{ |y| |y_2|}-\frac{\left\langle u,y\right\rangle |y_2|}{{|y|}^3}\right)
\\
& \qquad +\frac{1}{2}\left(\frac{\left\langle {[z,y]}_{{\mathfrak{m}}_2},u_2\right\rangle }{ |y|  |y_2| }-\frac{\left\langle {[z,y]}_{\mathfrak{m}},u\right\rangle  |y_2| }{{|y|}^3}\right)\left(\frac{ \llaa  v_2,y_2 \rraa    |y|}{|y_2|}-\frac{ \llaa   v,y \rraa    |y_2|}{|y|}\right)\Bigg].
\end{align*}

We note by assumption of the natural reducibility that $\langle {[z,v]}_{\mathfrak{m}},u\rangle =-\langle {[z,u]}_{\mathfrak{m}},v\rangle $ and $\langle {[z,y]}_{\mathfrak{m}},y\rangle =0$. On the other hand, from the condition ${\mathfrak{g}}'\subseteq \mathfrak{h}+{\mathfrak{m}}_1$ , one can deduce that $\langle {[z,y]}_{{\mathfrak{m}}_2},y_2\rangle =0$. Therefore, with some lengthy calculations and simplification, we have:

\begin{align*}
g_y{}&{}\left({[z,u]}_{\mathfrak{m}},v\right)+g_y\left(u,{[z,v]}_{\mathfrak{m}}\right)+2C_y\left({[z,y]}_{\mathfrak{m}},u,v\right)=
\\
&  \phi (b)\phi'(b)\left[\frac{-\left\langle {[z,y]}_{{\mathfrak{m}}_2},u_2\right\rangle  \llaa   v,y \rraa    }{ |y|  |y_2| }-\frac{\left\langle {[z,y]}_{\mathfrak{m}},u\right\rangle  \llaa  v_2,y_2 \rraa    }{ |y|  |y_2| }+\frac{\left\langle {[z,y]}_{\mathfrak{m}},u\right\rangle \left\langle y,v\right\rangle  |y_2| }{{|y|}^3} \right.
\\
& \qquad \left. +\frac{\left\langle {[z,u]}_{{\mathfrak{m}}_2},v_2\right\rangle  |y| }{|y_2|}-\frac{\left\langle {[z,u]}_{\mathfrak{m}},v\right\rangle  |y_2| }{|y|}  +\frac{\left\langle {[z,y]}_{{\mathfrak{m}}_2},u_2\right\rangle  \llaa  v_2,y_2 \rraa     |y| }{{|y_2|}^3}\right]
\\
& +\phi (b)\phi'(b) \left[\frac{-\left\langle {[z,y]}_{{\mathfrak{m}}_2},v_2\right\rangle \left\langle u,y\right\rangle }{ |y|  |y_2| }-\frac{\left\langle {[z,y]}_{\mathfrak{m}},v\right\rangle \left\langle u_2,y_2\right\rangle }{ |y|  |y_2| }+\frac{\left\langle {[z,y]}_{\mathfrak{m}},v\right\rangle \left\langle u,v\right\rangle  |y_2| }{{|y|}^3}
\right.
\\
& \qquad \left. -\frac{\left\langle {[z,v]}_{{\mathfrak{m}}_2},u_2\right\rangle  |y| }{|y_2|}+\frac{\left\langle {[z,y]}_{{\mathfrak{m}}_2},v_2\right\rangle \left\langle u_2,y_2\right\rangle  |y| }{{|y_2|}^3}\right]
\\
& +\phi (b)\phi'(b)\left[\frac{2\left\langle {[z,y]}_{\mathfrak{m}},v\right\rangle \left\langle u_2,y_2\right\rangle }{ |y|  |y_2| }-\frac{2\left\langle {[z,y]}_{\mathfrak{m}},v\right\rangle \left\langle u,y\right\rangle }{{|y|}^3}+\frac{2\left\langle {[z,y]}_{{\mathfrak{m}}_2},u_2\right\rangle  \llaa   v,y \rraa    }{ |y|  |y_2| } \right.
\\
& \qquad \left.-\frac{2\left\langle {[z,y]}_{\mathfrak{m}},u\right\rangle  \llaa   v,y \rraa     |y_2| }{{|y|}^3}+\frac{\left\langle {[z,y]}_{\mathfrak{m}},u\right\rangle  \llaa  v_2,y_2 \rraa    }{ |y|  |y_2| }+\frac{\left\langle {[z,y]}_{{\mathfrak{m}}_2},v_2\right\rangle \left\langle u,y\right\rangle }{ |y|  |y_2| }\right.
\\
&\qquad  \left.-\frac{\left\langle {[z,y]}_{{\mathfrak{m}}_2},u_2\right\rangle  \llaa   v,y \rraa    }{ |y|  |y_2| }-\frac{\left\langle {[z,y]}_{\mathfrak{m}},v\right\rangle \left\langle u_2,y_2\right\rangle }{ |y|  |y_2| }-\frac{\left\langle {[z,y]}_{{\mathfrak{m}}_2},u_2\right\rangle  \llaa  v_2,y_2 \rraa     |y| }{{|y_2|}^3}
\right.
\\
&\qquad  \left.-\frac{\left\langle {[z,y]}_{{\mathfrak{m}}_2},v_2\right\rangle \left\langle u_2,y_2\right\rangle  |y| }{{|y_2|}^3}+\frac{\left\langle {[z,y]}_{\mathfrak{m}},u\right\rangle  \llaa   v,y \rraa     |y_2| }{{|y|}^3}+\frac{\left\langle {[z,y]}_{\mathfrak{m}},v\right\rangle \left\langle u,y\right\rangle  |y_2| }{{|y|}^3}\right]
\\
& +\left({\phi'}^2(b)+\phi (b)\phi''(b)\right)\left(-\frac{\left\langle {[z,y]}_{{\mathfrak{m}}_2},u_2\right\rangle }{ |y|  |y_2| }+\frac{\left\langle {[z,y]}_{\mathfrak{m}},u\right\rangle  |y_2| }{{|y|}^3}\right)\left(\frac{ \llaa  v_2,y_2 \rraa    }{|y_2|}-\frac{ \llaa   v,y \rraa     |y_2| }{|y|}\right)
\\
& +\left({\phi'}^2(b)+\phi (b)\phi''(b)\right)\left(-\frac{\left\langle {[z,y]}_{{\mathfrak{m}}_2},v_2\right\rangle }{ |y|  |y_2| }+\frac{\left\langle {[z,y]}_{\mathfrak{m}},v\right\rangle  |y_2| }{{|y|}^3}\right)\left(\frac{\left\langle u_2,y_2\right\rangle  |y| }{|y_2|}-\frac{\left\langle u,y\right\rangle  |y_2| }{|y|}\right)
\\
& +\left({\phi'}^2(b)+\phi (b)\phi''(b)\right)\left(\frac{\left\langle {[z,y]}_{{\mathfrak{m}}_2},v_2\right\rangle  |y| }{|y_2|}+\frac{\left\langle {[z,y]}_{\mathfrak{m}},v\right\rangle  |y_2| }{|y|}\right)\left(\frac{\left\langle u_2,y_2\right\rangle }{ |y|  |y_2| }-\frac{\left\langle u,y\right\rangle  |y_2| }{{|y|}^3}\right)
\\
& +\left({\phi'}^2(b)+\phi (b)\phi''(b)\right)\left(\frac{\left\langle {[z,y]}_{{\mathfrak{m}}_2},u_2\right\rangle }{ |y|  |y_2| }-\frac{\left\langle {[z,y]}_{\mathfrak{m}},u\right\rangle  |y_2| }{{|y|}^3}\right)\left(\frac{ \llaa  v_2,y_2 \rraa    }{|y_2|}-\frac{ \llaa   v,y \rraa     |y_2| }{|y|}\right)
\\
& =\phi (b)\phi'(b)\left(\frac{\left\langle {[z,u]}_{{\mathfrak{m}}_2},v_2\right\rangle  |y| }{|y_2|}-\frac{\left\langle {[z,u]}_{\mathfrak{m}},v\right\rangle  |y_2| }{|y|}+\frac{\left\langle {[z,v]}_{{\mathfrak{m}}_2},u_2\right\rangle  |y| }{|y_2|}-\frac{\left\langle {[z,v]}_{\mathfrak{m}},u\right\rangle  |y_2| }{|y|}\right)
\\
& +\left({\phi'}^2(b)+\phi (b)\phi''(b)\right)\!\!\left(\!-\frac{\left\langle {[z,y]}_{{\mathfrak{m}}_2},v_2\right\rangle \left\langle u_2,y_2\right\rangle }{{|y_2|}^2}\!+\!\frac{\left\langle {[z,y]}_{{\mathfrak{m}}_2},v_2\right\rangle \left\langle u,y\right\rangle }{{|y|}^2}+\frac{\left\langle {[z,y]}_{\mathfrak{m}},v\right\rangle \left\langle u_2,y_2\right\rangle }{{|y|}^2} \right.
\\
&\quad \left. -\frac{\left\langle {[z,y]}_{\mathfrak{m}},v\right\rangle \left\langle u,y\right\rangle {|y_2|}^2}{{|y|}^4}\!\right)
\!\!\! \mleft(\frac{\left\langle {[z,y]}_{{\mathfrak{m}}_2},v_2\right\rangle\! \left\langle u_2,y_2\right\rangle }{{|y_2|}^2} \!-\! \frac{\left\langle {[z,y]}_{{\mathfrak{m}}_2},v_2\right\rangle \left\langle u,y\right\rangle }{{|y|}^2}\!-\!\frac{\left\langle {[z,y]}_{\mathfrak{m}},v\right\rangle \left\langle u_2,y_2\right\rangle }{{|y|}^2}\! \mright.
\\
& \qquad \left. +\frac{\left\langle {[z,y]}_{\mathfrak{m}},v\right\rangle \left\langle u,y\right\rangle {|y_2|}^2}{{|y|}^4}\right)\\
& =0.            \qedhere
\end{align*}
\end{proof}

\begin{remark}
\label{t.r1b}
Let $(G,g)$ be naturally reductive and consider $H= \{e\}$ (In other words, $\mathfrak{h}=\left\{0\right\}$). Then, if ${\mathfrak{g}}'\subseteq {\mathfrak{m}}_1$ then ${\mathfrak{m}}_2\subseteq Z(\mathfrak{g})$.
\end{remark}
\begin{proof}
If the above assumptions hold, then $(M, F)$ is naturally reductive. Now, using Lemma 3.2 in \cite{21}, we have $[{\mathfrak{m}}_2,{\mathfrak{m}}_2]\subseteq {\mathfrak{m}}_1$and $[{\mathfrak{m}}_2,{\mathfrak{m}}_2]\subseteq {\mathfrak{m}}_2$. Therefore, $\left[{\mathfrak{m}}_2,{\mathfrak{m}}_2\right]=0$. On the other hand, the second part of the same Lemma shows that $\left[{\mathfrak{m}}_1,{\mathfrak{m}}_2\right]\subseteq \mathfrak{h}=\left\{0\right\}$. Thus ${\mathfrak{m}}_2\subseteq Z(\mathfrak{g})$.
\end{proof}
\begin{remark}
\label{t.r.2b}
As a result, if $g$ is a bi-invariant Riemannian metric on a connected Lie group $G$ such that in the reductive decomposition $\mathfrak{g}=\mathfrak{m}={\mathfrak{m}}_1+{\mathfrak{m}}_2$, we have ${\mathfrak{g}}'\subseteq {\mathfrak{m}}_1$, then $(G,F)$ is naturally reductive.
\end{remark}

We recall that a vector field $X\in \mathfrak{g}\setminus \{0\}$ for a homogeneous Riemannian manifold $(M=G/H,g)$ (resp. a homogeneous Finsler manifold $(M=G/H,F)$) is called a geodesic vector if the curve $\sigma (t)={{\exp \left(tX\right) }}_o$ is a geodesic on $(G/H,g)$ (resp. on $(G/H,F)$). For a homogeneous Riemannian manifold $(G/H,g)$ with a reductive decomposition $\mathfrak{g}=\mathfrak{m}+ \mathfrak{h}$, a vector $X\in \mathfrak{g}$ is a geodesic vector if and only if $g\left(X_{\mathfrak{m}},{\left[X,Y\right]}_{\mathfrak{m}}\right)=0$ for all $Y\in \mathfrak{m}$ (see \cite{16}). It can be shown that, in a similar way for a homogeneous Finsler manifold $(M=G/H,F)$ (see \cite{17}), a vector $X\in \mathfrak{g}$ is a geodesic vector if and only if
\[g_{X_{\mathfrak{m}}}\left(X_{\mathfrak{m}},{\left[X,Z\right]}_{\mathfrak{m}}\right)=0,     \qquad    \forall Z\in \mathfrak{g}.\]

In \cite{19}, the relationship between geodesic vector fields on Riemannian and Finslerian homogeneous spaces in the case of an $(\alpha,\beta )$-metric is demonstrated. The following proposition shows the relationship between geodesic vector fields on these spaces in the case of an $(\alpha_1, \alpha_2)$-metric.
\begin{prop}
\label{t4.2}
Let $(M=G/H, F)$ be a homogeneous Finsler manifold where $F=\alpha \phi (\frac{\alpha _2}{\alpha})$ is an invariant $(\alpha_1, \alpha_2)$-metric defined by an invariant Riemannian metric $\alpha =\sqrt{g(\cdot,\cdot)}$ such that ${\mathfrak{g}}'\subseteq \mathfrak{h}+{\mathfrak{m}}_1$. If $X\in \mathfrak{g}\setminus \{0\}$ is a geodesic vector field for $(G/H,g)$, then $X$ is a geodesic vector field for $(G/H,F)$. Conversely, if $X$ is a geodesic vector field for $(G/H,F)$ such that $\left(\phi^2(b)-\phi (b)\phi'(b)b\right)\neq 0$ and $\phi''(b)\le 0$, where $b=\frac{{\left|X_{{\mathfrak{m}}_2}\right|}_g}{{\left|X_{\mathfrak{m}}\right|}_g}$, then $X$ is a geodesic vector field for $(G/H,g)$.
\end{prop}
\begin{proof}
Let $X$ be a geodesic vector field for $(G/H,g)$. Given the assumption ${\mathfrak{g}}'\subseteq \mathfrak{h}+{\mathfrak{m}}_1$ , it can be said $\left\langle {X_{\mathfrak{m}}}_2,{\left[X,Z\right]}_{{\mathfrak{m}}_2}\right\rangle =0$. Therefore, we have
\begin{align*}
g_{X_{\mathfrak{m}}}
{}&{} \left(X_{\mathfrak{m}},{\left[X,Z\right]}_{\mathfrak{m}}\right)=\left\langle X_{\mathfrak{m}},{\left[X,Z\right]}_{\mathfrak{m}}\right\rangle \phi^2(b) \\
& +\phi (b)\phi'(b)\left(\frac{\left\langle X_{\mathfrak{m}},{\left[X,Z\right]}_{\mathfrak{m}}\right\rangle \left\langle X_{{\mathfrak{m}}_2},X_{{\mathfrak{m}}_2}\right\rangle }{\left|X_{\mathfrak{m}}\right|\left|X_{{\mathfrak{m}}_2}\right|}+\frac{\left\langle {X_{\mathfrak{m}}}_2,{\left[X,Z\right]}_{{\mathfrak{m}}_2}\right\rangle \left\langle X_{\mathfrak{m}},X_{\mathfrak{m}}\right\rangle }{\left|X_{\mathfrak{m}}\right|\left|X_{{\mathfrak{m}}_2}\right|} \right.
\\
& \quad  -\frac{\left\langle X_{\mathfrak{m}},X_{\mathfrak{m}}\right\rangle \left\langle X_{\mathfrak{m}},{\left[X,Z\right]}_{\mathfrak{m}}\right\rangle \left|{X_{\mathfrak{m}}}_2\right|}{{\left|X_{\mathfrak{m}}\right|}^3}+\frac{\left\langle {X_{\mathfrak{m}}}_2,{\left[X,Z\right]}_{{\mathfrak{m}}_2}\right\rangle \left|X_{\mathfrak{m}}\right|}{\left|X_{{\mathfrak{m}}_2}\right|}-\frac{\left\langle {X_{\mathfrak{m}}}_2,{\left[X,Z\right]}_{{\mathfrak{m}}_2}\right\rangle \left|X_{{\mathfrak{m}}_2}\right|}{\left|X_{\mathfrak{m}}\right|}
\\
& \quad \left. -\frac{\left\langle X_{{\mathfrak{m}}_2},X_{{\mathfrak{m}}_2}\right\rangle \left\langle {X_{\mathfrak{m}}}_2,{\left[X,Z\right]}_{{\mathfrak{m}}_2}\right\rangle \left|X_{\mathfrak{m}}\right|}{{\left|X_{{\mathfrak{m}}_2}\right|}^3}\right)
\\
& +\left(\phi^{'^2}(b)+\phi \left(b\right)\phi''(b)\right)\left(\frac{\left\langle X_{{\mathfrak{m}}_2},X_{{\mathfrak{m}}_2}\right\rangle }{\left|X_{\mathfrak{m}}\right|\left|X_{{\mathfrak{m}}_2}\right|}-\frac{\left\langle X_{\mathfrak{m}},X_{\mathfrak{m}}\right\rangle \left|X_{{\mathfrak{m}}_2}\right|}{{\left|X_{\mathfrak{m}}\right|}^3}\right)
\\
& \quad \left(\frac{\left\langle {\left[X,Z\right]}_{{\mathfrak{m}}_2},{X_{\mathfrak{m}}}_2\right\rangle }{\left|X_{{\mathfrak{m}}_2}\right|}-\frac{\left\langle {\left[X,Z\right]}_{\mathfrak{m}},{X_{\mathfrak{m}}}_2\right\rangle \left|X_{{\mathfrak{m}}_2}\right|}{\left|X_{\mathfrak{m}}\right|}\right)
\\
& =\left\langle X_{\mathfrak{m}},{\left[X,Z\right]}_{\mathfrak{m}}\right\rangle \phi^2(b)+\phi (b)\phi'(b)\left(\frac{\left\langle {X_{\mathfrak{m}}}_2,{\left[X,Z\right]}_{{\mathfrak{m}}_2}\right\rangle \left|X_{\mathfrak{m}}\right|}{|X_{{\mathfrak{m}}_2}|}-\frac{\left\langle X_{\mathfrak{m}},{\left[X,Z\right]}_{\mathfrak{m}}\right\rangle |X_{{\mathfrak{m}}_2}|}{\left|X_{\mathfrak{m}}\right|}\right),
\intertext{
which results in $g_{X_{\mathfrak{m}}}\left(X_{\mathfrak{m}},{\left[X,Z\right]}_{\mathfrak{m}}\right)=0$. To prove the converse, we note that the last sentence in above can be written as follows:}
&\left\langle X_{\mathfrak{m}},{\left[X,Z\right]}_{\mathfrak{m}}\right\rangle \left(\phi^2(b)-\phi (b)\phi'(b)\frac{|X_{{\mathfrak{m}}_2}|}{\left|X_{\mathfrak{m}}\right|}\right)+\left\langle {X_{\mathfrak{m}}}_2,{\left[X,Z\right]}_{{\mathfrak{m}}_2}\right\rangle \phi (b)\phi'(b)\frac{\left|X_{\mathfrak{m}}\right|}{|X_{{\mathfrak{m}}_2}|}.
\end{align*}
Therefore, given the condition in the assumption and that $\left\langle {X_{\mathfrak{m}}}_2,{\left[X,Z\right]}_{{\mathfrak{m}}_2}\right\rangle =0$, the statement holds.
\end{proof}
\begin{remark}\label{t.r.3}
In the above proposition, we observe that if $\phi''(b)\le 0$, $\phi'(b)\neq 0$, and $\phi (b)>0$, and also considering Theorem~\ref{t2.26}, we have $\phi^2(b)-\phi (b)\phi'(b)\frac{\left|X_{{\mathfrak{m}}_2}\right|}{\left|X_{\mathfrak{m}}\right|}\neq 0$. So, the last condition in the above proposition can be replaced by the conditions $\phi'(b)\neq 0$ and $\phi''(b)\le 0$.
\end{remark}

In the remainder of this section, we analyze the $(\alpha_1,\alpha_2)$-metrics on the tangent bundle of a Lie group under certain conditions. We first present the following remark.

\begin{remark}\label{t5.1}
Let $G$ be a Lie group equipped with a bi-invariant Riemannian metric $g$. Then the Riemannian metric $\tilde{g}$, as defined in Definition \ref{t2.15}, on the tangent bundle $TG$ is bi-invariant if and only if $G$ is commutative.
\end{remark}
\begin{proof}
Let $g\left(\left[X,Y\right],Z\right)+g\left(X,\left[Z,Y\right]\right)=0$ for every $X,Y,Z\in \mathfrak{g}$. To show that the statement holds, it suffices to note that
\begin{align*}
&\tilde{g}\left(\left[X^c,Y^v\right],Z^v\right)+\tilde{g}\left(X^c,\left[Z^v,Y^v\right]\right)=\tilde{g}\left({\left[X,Y\right]}^v,Z^v\right)=g\left(\left[X,Y\right],Z\right),\\
&\tilde{g}\left(\left[X^v,Y^v\right],Z^c\right)+\tilde{g}\left(X^v,\left[Z^c,Y^v\right]\right)=\tilde{g}\left(X^v,{\left[Z,Y\right]}^v\right)=g\left(X,\left[Z,Y\right]\right).
\end{align*}
Furthermore, for the remaining cases, the sum vanishes. Consequently, the equality holds if and only if $\mathfrak{g}$ is an Abelian Lie algebra.
\end{proof}

Here we give a procedure to construct an $(\alpha _1,\alpha _2)$-metric on the tangent bundle $TG$ using a Randers metric on the Lie group $G$. Let $F(x,y)=\alpha (x,y)+\beta (x,y)=\sqrt{g_x(y,y)}+g(X\left(x\right),y)$ be a left-invariant Randers metric on an n-dimensional connected Lie group $G$,  where $x\in M$, $0\neq y\in T_xM$,  $g$ is a left-invariant Riemannian metric and $X$ is a left-invariant vector field  (${\left\|X\right\|}_\alpha <1$). Then, there exists an $\alpha$-orthogonal decomposition $T_eG=\mathfrak{g}={\mathcal{V}}_1\oplus {\mathcal{V}}_2$ $\left({\mathcal{V}}_1{\bot}_g{\mathcal{V}}_2\right)$ in which ${\dim {\mathcal{V}}_1=n-1}$ and ${\dim {\mathcal{V}}_2=1 }$.

Let ${\mathcal{V}}_1=\spann  \left\{Y_1,Y_2,\ldots ,Y_{n-1}\right\}$ and ${\mathcal{V}}_2=\Spann \left\{X\right\}$ so that ${\left\{Y_i\right\}}^{n-1}_{i=1}$ are orthonormal with respect to the Riemannian metric $\alpha $ and also $X$ be $\alpha $-orthogonal to ${\left\{Y_i\right\}}^{n-1}_{i=1}$. With the lift from $G$ to the tangent bundle $TG$ by the vertical and complete vector fields, there exists an $\widetilde\alpha $-orthogonal decomposition $T_{\left(e,y\right)}TG=\widetilde{\mathfrak{g}}=\widetilde{{\mathcal{V}}_1}\oplus \widetilde{{\mathcal{V}}_2}$ $(\widetilde{{\mathcal{V}}_1}{\bot }_{\tilde{g}}\widetilde{{\mathcal{V}}_2})$ such that $\widetilde{{\mathcal{V}}_1}=\Spann \left\{Y^c_1,Y^c_2,\ldots ,Y^c_{n-1},Y^v_1,Y^v_2,\ldots ,Y^v_{n-1}\right\}$ and $\widetilde{{\mathcal{V}}_2}=\Spann \left\{X^c,X^v\right\}$, where $\widetilde\alpha =\sqrt{\tilde{g}(\cdot,\cdot)}$ and $\tilde{g}$ is the left-invariant lifted Riemannian metric on $TG$ according to Definition \ref{t2.15}.

It can be seen that if $\tilde{Z}\in \widetilde{\mathfrak{g}}$, then
\[\exists    \underset{1\le i\le n-1}{{\mu }^c_i, {\mu }^v_i,{\lambda }^c,{\lambda }^v}; \quad    \tilde{Z}=\left(\sum^{n-1}_{i=1}{\left({\mu }^c_iY^c_i+{\mu }^v_iY^v_i\right)}\right)+{\lambda }^cX^c+{\lambda }^vX^v\]
Therefore, the $(\alpha _1,\alpha _2)$-metric on the tangent bundle $TG$ is as follows:
\begin{align*}
\tilde{F}\left((x,y),\tilde{Z}\right)
& =\widetilde\alpha \phi \left(\frac{\widetilde{\alpha _2}}{\widetilde\alpha }\right)={\left\|\tilde{Z}\right\|}_{\widetilde\alpha }\phi \left(\frac{{\left\|{\lambda }^cX^c+{\lambda }^vX^v\right\|}_{\widetilde{\alpha _2}}}{{\left\|\tilde{Z}\right\|}_{\widetilde\alpha }}\right)\\
&  =\widetilde\alpha \mathrm{\psi }\left(\frac{\widetilde{\alpha _1}}{\widetilde\alpha }\right)={\left\|\tilde{Z}\right\|}_{\widetilde\alpha }\mathrm{\psi }\left(\frac{{\left\|\sum^{n-1}_{i=1}{({\mu }^c_iY^c_i+{\mu }^v_iY^v_i)}\right\|}_{\widetilde{\alpha _1}}}{{\left\|\tilde{Z}\right\|}_{\widetilde\alpha }}\right).
\end{align*}
Here we recall the concept of S-curvature. Let $(M, F)$ be a $n$-dimensional Finsler manifold and $\left\{b_i\right\}$ be a basis for $T_xM$.
For a vector $y\in T_xM\setminus \{0\}$, let $\gamma(t)$ be a geodesic with conditions $\gamma(0)=x$ and $\gamma'(0)=y$.
The S-curvature, first introduced by Z. Shen to measure the rate of change of the volume form of a Finsler space along geodesics, is a function on $TM\setminus \{0\}$ defined as follows:
\[S(x,y)\colonequals \frac{d}{dt}\left[\tau \left(\gamma(t),\gamma'(t)\right)\right]_{t=0},\]
where $\tau =\tau(y)$ is called the distortion function associated with $F$, which is independent of the choice of the basis $\{b_i\}$ and also is homogeneous of degree zero with respect to $y$, and is defined as $\tau(y)\colonequals Ln\frac{\sqrt{\det (g_{ij}(y)}}{{\sigma }_F}$.
It should be noted that in the latter formula, ${\sigma}_F$ is defined as ${\sigma}_F\left(x\right)\colonequals \frac{Vol(B^n (1) )}{Vol(B^n_x)}$, where $Vol$ denotes the volume of a subset in Euclidean space, ${\mathbb{R}}^n$, $B^n(1)$ is an open ball of radius $1$, and $B^n_x\colonequals \left\{(y^i\in {\mathbb{R}}^n|F(y^ib_i)<1\right\}$ is a strongly convex open set of ${\mathbb{R}}^n$.

\begin{prop}
\label{t5.4}
Let $G$ be a nilpotent connected Lie group, and $F$ be a Randers metric generated by a left-invariant Riemannian metric $g$ and a left-invariant vector field $X$. If $S_F=0$, then, $S_{\tilde{F}}=0$, where $S_F$ and $S_{\tilde{F}}$ are the $S$-curvature of the Randers metric $F$ on $G$ and the $S$-curvature of the $(\alpha_1, \alpha_2)$-metric $\tilde{F}$ on $TG$ constructed by the Randers metric F, respectively.
\end{prop}
\begin{proof}
According to Proposition 7.6 in Chapter 7 of \cite{8}
\[S_F=0\Leftrightarrow X\in Z\left(\mathfrak{g}\right)\Leftrightarrow X^c, X^v\in Z(\widetilde{\mathfrak{g}}).\]
On the other hand, $\widetilde{{\mathcal{V}}_2}=\spann \left\{X^c,X^v\right\}$. Therefore, according to Corollary 4.4 in  \cite{11}, the statement holds.
\end{proof}

\begin{prop}\label{t5.6}
Let $G$ be an arbitrary connected Lie group and $S_{\tilde{F}}=0$. then, $S_F=0$ and the base Randers metric $F$ is of Berwald type.
\end{prop}
\begin{proof}
According to Corollary~4.4 in \cite{10},
\[  \begin{cases}
\tilde{g}\left({[Y}^c_i,X^c\right] , Y^c_i)=0  \\
\tilde{g}\left({[Y}^v_i,X^c\right] , Y^v_i)=0
\end{cases}  \Longrightarrow g\left(\left[Y_i,X\right] , Y_i\right)=0.                        \tag{$\star$}\]
Also,
\[\begin{cases}
\tilde{g}\left({[Y}^c_i,X^c\right] , X^c)=0  \\
\tilde{g}\left({[Y}^v_i,X^v\right] , X^v)=0
\end{cases}
\Longrightarrow g\left(\left[Y_i,X\right] , X\right)=0.                       \tag{$\star\star$}\]
Therefore, from $(\star)$ and $(\star\star)$, it follows that $S_F=0$. We also note that $(\star\star)$ implies that $F$ is of Douglas type, and according to Proposition~7.4 in \cite{8}, $F$ is of Berwald type.
\end{proof}



\begin{thebibliography}{99}
\bibitem{1}
 I. Agricola, A.C. Ferreira, Tangent Lie groups are Riemannian naturally reductive spaces, \textit{Adv. Appl. Clifford Algebr.} \textbf{27} (2017), 895-911.

 \bibitem{2}
 H. An, Z. Yan, and Sh. Zhang, S-Curvature and geodesic orbit property of invariant $(\alpha _1,\alpha _2)$ - metrics on spheres, \textit{Bull. Korean Math. Soc.} \textbf{60} (2023), No. 1, pp. 33-46.

 \bibitem{3}
 F. Asgari, H. R. Salimi Moghaddam, On the Riemannian geometry of tangent Lie groups, \textit{Rend. Circ. Mat. Palermo}, II. Ser \textbf{67} (2018), pp. 185-195.

 \bibitem{5}
 D. Bao, S. S. Chern, and Z. Shen, \textit{An introduction to Riemannian--Finsler geometry}, Grad. Texts in Math., vol. 200, Springer, New York, 2000.

 \bibitem{6}
 S. S. Chern, Z.Shen, \textit{Riemann-Finsler Geometry}. Singapore: World Scientific Publishers, 2004.

 \bibitem{8}
 S. Deng, \textit{Homogeneous Finsler Spaces}. New York: Springer, 2012.

 \bibitem{9}
 S. Deng and Z. Hou, Invariant Finsler metrics on homogeneous manifolds, \textit{J. Phys. A: Math. Gen.} \textbf{37} (2004) 8245-8253.

 \bibitem{10}
 S. Deng and Z. Hou, Naturally reductive homogeneous Finsler spaces, \textit{Manuscripta Math.} \textbf{131}(\textbf{1-2}) (2010) 215-229.

 \bibitem{11}
 S. Deng, M. Xu, $(\alpha _1,\alpha _{2})$-metrics and Clifford-Wolf homogeneity. \textit{J. Geom. Anal.} \textbf{26}(\textbf{3}) (2016), 2282--2321.

 \bibitem{12}
 S.Deng, M. Xu, Clifford--Wolf translations of Finsler spaces. \textit{Forum Math.} \textbf{26} (2014), 1413--1428.

 \bibitem{13}
 L. Huang, On the fundamental equations of homogeneous Finsler spaces, \textit{Differential Geom. Appl.} \textbf{40} (2015), 187--208.

 \bibitem{14}
 L. Huang and X. Mo, Homogeneous Einstein metrics on (4n + 3)-dimensional spheres, \textit{Canad. Math. Bull.} \textbf{62} (2018), no. 3, 509--523.

 \bibitem{16}
 Kowalski, O. and Vanhecke, L., Riemannian manifolds with homogeneous geodesics, \textit{Boll. Unione Mat. Ital.} \textbf{5}(1991), 189246.

 \bibitem{17}
 D. Latifi, Homogeneous geodesics in homogeneous Finsler spaces, \textit{J. Geom. Phys.} \textbf{57}(2007) 1421-1433.

 \bibitem{19}
 M. Parhizkar and H. R. Salimi Moghaddam, Geodesic Vector fields of invariant $(\alpha ,\beta )$-metrics on Homogeneous spaces, \textit{Int. Electron. J. Geom. }\textbf{6} (2013), 39--44.

 \bibitem{21}
 J. Tan, M. Xu, Naturally reductive $(\alpha _1,\alpha _{2})$-metrics, \textit{Acta Math. Sci.} \textbf{43} (2023), 1547-1560.

 \bibitem{22}
 M. Xu, S. Deng, Clifford--Wolf homogeneous Randers spaces. \textit{J. Lie Theory} \textbf{23}(2013), 837--845.

 \bibitem{23}
 M. Xu, S. Deng, Left invariant Clifford--Wolf homogeneous ($\alpha$, $\beta$)-metrics on compact semisimple Lie groups. \textit{Transform. Groups} \textbf{20} (2015), 395-416.

 \bibitem{24}
 L. Zhang and M. Xu, Standard homogeneous $(\alpha _1,\alpha _{2})$ -metrics and geodesic orbit property. \textit{Math. Nachr.} \textbf{295}(7) (2022), 1443-1453.

 \bibitem{25}
 S. Zhang, Z. Yan and S. Deng, Naturally reductive homogeneous $(\alpha ,\beta )$ spaces, \textit{Publ. Math. Debrecen} \textbf{102}(3-4) (2023), 415-427.

\end{thebibliography}
\end{document}